\documentclass[a4paper,10pt]{amsart}

\usepackage{amsmath}
\usepackage{amsthm}
\usepackage{amssymb}
\usepackage{amsfonts}
\usepackage{mathrsfs}  
\usepackage{enumitem} 
\usepackage{xcolor}
\usepackage{soul}
\usepackage[bookmarks=false,hyperindex,pdftex,colorlinks,citecolor=blue,urlcolor=cyan]{hyperref}
\usepackage[a4paper,lmargin=3cm,rmargin=3cm,tmargin=4cm,bmargin=4cm,marginparwidth=2.8cm,marginparsep=1mm]{geometry} 

\DeclareMathOperator{\lspan}{span}                          
\DeclareMathOperator{\diam}{diam}                           
\DeclareMathOperator{\Lip}{Lip}                             

\newcommand{\NN}{\mathbb{N}}                                
\newcommand{\ZZ}{\mathbb{Z}}                                
\newcommand{\RR}{\mathbb{R}}                                
\newcommand{\ep}{\varepsilon}

\newcommand{\abs}[1]{\left|{#1}\right|}                     
\newcommand{\pare}[1]{\left({#1}\right)}                    
\newcommand{\set}[1]{\left\{{#1}\right\}}                   
\newcommand{\cl}[1]{\overline{#1}}                          
\newcommand{\compl}{\stackrel{c}{\hookrightarrow}}          

\newcommand{\lipfree}[1]{\mathcal{F}({#1})}                 

\renewcommand{\leq}{\leqslant}
\renewcommand{\geq}{\geqslant}

\theoremstyle{plain}
\newtheorem{theorem}{Theorem}

\newtheorem{proposition}[theorem]{Proposition}

\newtheorem*{claim*}{Claim}

\theoremstyle{definition}
\newtheorem*{definition*}{Definition}

\theoremstyle{remark}

\begin{document}

\title{Lipschitz-free spaces over non-porous subsets of $\mathbb{R}^n$}

\author[R. J. Aliaga]{Ram\'on J. Aliaga}
\address[R. J. Aliaga]{Instituto Universitario de Matem\'atica Pura y Aplicada,
Universitat Polit\`ecnica de Val\`encia,
Camino de Vera S/N,
46022 Valencia, Spain}
\email{ramon.aliaga@upv.es}

\date{} 


\begin{abstract}
We prove that the Lipschitz-free space $\mathcal{F}(M)$ contains a complemented copy of $\mathcal{F}(\mathbb{Z}^n)$ whenever $M\subset\mathbb{R}^n$ is not porous. Consequently, if $M\subset\mathbb{R}^n$ is uniformly discrete and not porous then $\mathcal{F}(M)$ is isomorphic to $\mathcal{F}(\mathbb{Z}^n)$.
\end{abstract}



\maketitle
\thispagestyle{empty} 


\section{Introduction and main result}

Let $(M,d)$ be a complete metric space and fix a base point $0\in M$. We define the \emph{Lipschitz space} $\Lip_0(M)$ as the Banach space of real-valued Lipschitz functions $f:M\to\RR$ such that $f(0)=0$, endowed with the (best) Lipschitz constant as the norm. This space has a canonical isometric predual, the \emph{Lipschitz-free space} $\lipfree{M}$, that can be obtained as follows: let $\delta(x):f\mapsto f(x)$ denote the evaluation functional on $x\in M$, which is an element of $\Lip_0(M)^*$, and set
$$
\lipfree{M} = \cl{\lspan}\set{\delta(x)\,:\,x\in M}.
$$
We recommend \cite{Weaver2} as a general reference on Lipschitz and Lipschitz-free spaces.

This note focuses on the isomorphic classification of Lipschitz-free spaces over subsets of $\RR^n$ (equivalently, of finite-dimensional Banach spaces). In that setting, Kaufmann proved in \cite{Kaufmann} that $\lipfree{M}$ is isomorphic to $\lipfree{\RR^n}$ whenever $M\subset\RR^n$ has non-empty interior. Attempts to improve on that result have been largely unsuccessful so far; for instance, it is still unknown if the same result holds under the weaker assumption that $M$ has positive $n$-dimensional Lebesgue measure. Perhaps the strongest follow-up to date is given in \cite{Aliaga26}, where it is shown that their duals $\Lip_0(M)$ and $\Lip_0(\RR^n)$ are isomorphic when $M$ is non-porous (hence, in particular, when it has positive measure). The corresponding statement fails for Lipschitz-free spaces, as e.g. $\ZZ^n$ is non-porous in $\RR^n$ but $\lipfree{\ZZ^n}$ is not isomorphic to $\lipfree{\RR^n}$ (the former has the Radon-Nikod\'ym property and the latter does not, by virtue of \cite[Theorem C]{AGPP}).

In this note, we address this issue and obtain information on $\lipfree{M}$, instead of $\Lip_0(M)$, for non-porous $M\subset\RR^n$. Our main result shows that the previous observation is sharp in some sense, as $\lipfree{\ZZ^n}$ is the ``smallest'' possible Lipschitz-free space for such $M$. It also provides a full identification when $M$ is uniformly discrete.

\begin{theorem}\label{th:porous_lipfree}
If $M\subset\RR^n$ is not porous then $\lipfree{M}$ contains a complemented subspace isomorphic to $\lipfree{\ZZ^n}$. If $M$ is moreover uniformly discrete then $\lipfree{M}$ is isomorphic to $\lipfree{\ZZ^n}$.
\end{theorem}

This extends the main result of \cite{Aliaga26}. As a particular case, we get that $\lipfree{E\times\stackrel{n}{\ldots}\times E}$ is isomorphic to $\lipfree{\ZZ^n}$ whenever $E\subset\RR$ is uniformly discrete and not porous, which in particular provides an affirmative answer to \cite[Question 6]{Aliaga26}. We refer to the very recent preprint \cite{Mason} for more results on Lipschitz-free spaces over products of the form $E_1\times\ldots\times E_n$ for certain classes of (possibly different) sets $E_i\subset\RR$.

\section{Notation and preliminaries}

In the sequel, $M$ is assumed to stand for a complete metric space with metric $d$, and the closed ball with center $x\in M$ and radius $r>0$ is denoted $B_M(x,r)$. We say that $M$ is \emph{uniformly discrete} if there is $\delta>0$ such that $d(x,y)\geq\delta$ for all $x\neq y\in M$. A subset $N\subset M$ is \emph{$\ep$-dense} in $M$ if for all $x\in M$ there is $y\in N$ such that $d(x,y)\leq\ep$. If $N$ is uniformly discrete and $\ep$-dense in $M$ for some $\ep>0$, then it is a \emph{net} in $M$. By an application of Zorn's lemma, every uniformly discrete subset of $M$ is a subset of a net in $M$.

We consider only real scalars in this note. We use the following Banach space theoretic notation: given Banach spaces $X,Y$, we write $X\equiv Y$ if they are linearly isometric, $X\sim Y$ if they are linearly isomorphic, and $X\compl Y$ if $Y$ contains a complemented subspace isomorphic to $X$. The $\ell_1$-sum of a family $\set{X_i:i\in I}$ of Banach spaces is denoted by $\pare{\bigoplus_{i\in I}X_i}_{\ell_1}$.

We recall that the isometric class of a Lipschitz-free space is invariant under changes of base point. We also have $\lipfree{M_1}\equiv\lipfree{M_2}$ if $M_1$ and $M_2$ are related by a surjective \emph{dilation}, i.e. a mapping $\varphi:M_1\to M_2$ such that $d(\varphi(x),\varphi(y))=C d(x,y)$ for some fixed constant $C$ and all $x,y\in M_1$. More generally, $\lipfree{M_1}$ and $\lipfree{M_2}$ are $D$-isomorphic (for $D\geq 1$) if $M_1$ and $M_2$ are \emph{bi-Lipschitz equivalent with distortion $D$}, i.e. there is a surjective mapping $\varphi:M_1\to M_2$ such that
$$
C\cdot d(x,y) \leq d(\varphi(x),\varphi(y)) \leq CD\cdot d(x,y)
$$
for some fixed constant $C$ and all $x,y\in M_1$.

Our arguments will use several auxiliary results that we collect here without proof. The first is Kalton's decomposition theorem, originally from \cite{Kalton}. For the proof of the version used here, see e.g. \cite[Lemma 1.2]{AliagaMedina}.

\begin{proposition}\label{pr:kalton}
For any metric space $M$ with base point $0\in M$,
$$
\lipfree{M}\compl\pare{\bigoplus_{n=1}^\infty \lipfree{B_M(0,2^n)}}_{\ell_1} .
$$
\end{proposition}

The next one establishes the existence of complementation relations for Lipschitz-free spaces over subsets of $\RR^n$ (or, more generally, of doubling metric spaces). For reference, see e.g. \cite[Fact 2.4]{AliagaMedina} and the discussion following its proof.

\begin{proposition}\label{pr:extension}
If $N\subset M\subset\RR^n$ then $\lipfree{N}\compl\lipfree{M}$.
\end{proposition}

For the next results, we recall the following notion from \cite{Aliaga26}. A family $\mathcal{S}$ of subsets of $M$ is \emph{well-separated} if there exist a point $x_0\in M$ and a constant $C<\infty$ such that
$$
d(x,x_0) + d(y,x_0) \leq C\cdot d(x,y)
$$
for $x,y$ belonging to different elements of $\mathcal{S}$. Well-separated families make it possible to decompose the Lipschitz-free space over a union as a sum of separate Lipschitz-free spaces; see \cite[Lemma 3.3]{Aliaga26} for a more precise statement.

\begin{proposition}\label{pr:separation}
If $\mathcal{S}$ is an infinite family of well-separated non-empty subsets of $M$, then
$$
\pare{\bigoplus_{A\in\mathcal{S}} \lipfree{A}}_{\ell_1} \compl \mathcal{F}\pare{\bigcup\mathcal{S}} .
$$
\end{proposition}

Finally, we recall the definition of porosity. Let $X$ be a metric space. A subset $M\subset X$ is \emph{porous (in $X$)} if there exists $\lambda\in (0,1)$ such that every ball $B_X(x,r)$ (with $r\leq\diam(X)$) contains a ball $B_X(y,\lambda r)$ of comparable size that does not intersect $M$. It is clear, by taking $r\to\infty$, that any net in a Banach space is non-porous. Then we have the following result, proved in \cite[Proposition 2.6]{Aliaga26}; part (a) is straightforward, its real content is part (b).

\begin{proposition}\label{pr:porous}
Let $X$ be a complete geodesic metric space (e.g. a Banach space). Suppose that $M$ is a non-porous subset of $X$. Then there exists a sequence of balls $B_n=B_X(p_n,r_n)$ in $X$ such that
\begin{enumerate}[label={\upshape{(\alph*)}}]
\item $M\cap B_n$ is $\ep_n r_n$-dense in $B_n$ for some sequence $\ep_n\searrow 0$ of positive numbers, and
\item $B_n$ are disjoint and well-separated.
\end{enumerate}
\end{proposition}

\section{Proof of the main result}

We start by stating and proving a general form of our main argument, modeled after (and improving on) \cite[Proposition 3.5]{Aliaga26}.

\begin{proposition}\label{pr:compl porous}
Let $X$ be a Banach space, $D$ a uniformly discrete subset of $X$, and $M$ a non-porous subset of $X$. Then there exists a subset $N\subset M$ such that $\lipfree{D}\compl\lipfree{N}$.
\end{proposition}

\begin{proof}
Let $R=\inf\set{d(x,y):x\neq y\in D}>0$. Assume without loss of generality that the origin $0$ belongs to $D$ and that we take it as the base point. Denote $D_n = D\cap B_X(0,2^n)$.

By Proposition \ref{pr:porous}, there exists a sequence of disjoint, well-separated balls $B_n=B_X(p_n,r_n)$ and a sequence $\ep_n\searrow 0$ such that $M\cap B_n$ is $\ep_nr_n$-dense in $B_n$. By passing to a subsequence, we may assume that $\ep_n\leq 2^{-(n+2)}R$.

Let $\varphi_n$ be the affine function mapping $B_n$ onto $B_X(0,2^n)$, i.e. $\varphi_n:x\mapsto 2^nr_n^{-1}(x-p_n)$, which is a dilation. Then $\varphi_n(M\cap B_n)$ is $\ep_n2^n$-dense, that is, $\frac{R}{4}$-dense in $B_X(0,2^n)$. Fix a mapping
$$
\psi_n:D_n\to\varphi_n(M\cap B_n)
$$
with $d(x,\psi_n(x))\leq\frac{R}{4}$ for all $x\in D_n$. Then, for $x\neq y\in D_n$ we have
$$
\frac{1}{2}d(x,y) \leq d(x,y)-\frac{R}{2} \leq d(\psi_n(x),\psi_n(y)) \leq d(x,y)+\frac{R}{2} \leq \frac{3}{2}d(x,y) .
$$
Thus $\psi_n$ is a bi-Lipschitz map from $D_n$ onto $\psi_n(D_n)$ with distortion (at most) $3$.

Let $N=\bigcup_{n=1}^\infty\varphi_n^{-1}(\psi_n(D_n))\subset M$. Then we have
$$
\lipfree{D} \compl \pare{\bigoplus_{n=1}^\infty\lipfree{D_n}}_{\ell_1} \sim \pare{\bigoplus_{n=1}^\infty\lipfree{\psi_n(D_n)}}_{\ell_1} \equiv \pare{\bigoplus_{n=1}^\infty\lipfree{\varphi_n^{-1}(\psi_n(D_n))}}_{\ell_1} \compl \lipfree{N} .
$$
Indeed: the first relation follows from Proposition \ref{pr:kalton}; the second from the fact that $\lipfree{D_n}\sim\lipfree{\psi_n(D_n)}$ with a uniform isomorphism constant for all $n$; the third from the fact that each $\varphi_n$ is a dilation; and the last one from Proposition \ref{pr:separation} as the sets $B_n$, hence also $\varphi_n^{-1}(\psi_n(D_n))\subset B_n$, are well-separated.
\end{proof}

The missing ingredient in the argument above is Proposition \ref{pr:extension}, that is only valid for $X=\RR^n$. With it, we can complete the proof of Theorem \ref{th:porous_lipfree} as follows.

\begin{proof}[Proof of Theorem \ref{th:porous_lipfree}]
Applying Proposition \ref{pr:compl porous} to $M$ with $X=\RR^n$ and $D=\ZZ^n$ yields a subset $N\subset M$ such that $\lipfree{\ZZ^n}\compl\lipfree{N}$. On the other hand, $\lipfree{N}\compl\lipfree{M}$ by Proposition \ref{pr:extension}. This establishes the first assertion.

For the second one, note that $\ZZ^n$ is a non-porous subset of $\RR^n$, thus Proposition \ref{pr:compl porous} yields $\lipfree{M}\compl\lipfree{N}$ for some $N\subset\ZZ^n$, and hence $\lipfree{M}\compl\lipfree{\ZZ^n}$ by Proposition \ref{pr:extension}. The statement now follows by Pe\l czy\'nski's decomposition method, as $\lipfree{\ZZ^n}$ is isomorphic to its $\ell_1$-sum by e.g. \cite[Theorem 3.6]{AliagaMedina}.
\end{proof}

We remark that our argument can be generalized to the situation where the ambient space is a general Carnot group in place of $\RR^n$. In that case, our main result reads as follows.

\begin{theorem}
Let $G$ be a Carnot group equipped with its Carnot-Carath\'eodory metric and let $N$ be a net in $G$. If $M\subset G$ is not porous then $\lipfree{N}\compl\lipfree{M}$. If $M$ is moreover uniformly discrete then $\lipfree{M}\sim\lipfree{N}$.
\end{theorem}

\noindent We omit the details of the proof, as it is identical to that of Theorem \ref{th:porous_lipfree} with only cosmetic changes in the definition of the dilations $\varphi_n$, in the same way as in the proof of \cite[Theorem 3.10]{Aliaga26}. Proposition \ref{pr:extension} is still valid with $G$ in place of $\RR^n$, and $\lipfree{N}$ is isomorphic to its countable $\ell_1$ sum by \cite[Corollary 5.10]{AACD3}.

\bigskip

We finish this note by noting that Theorem \ref{th:porous_lipfree} provides a solution to \cite[Question 6]{Aliaga26}, which asks whether the Lipschitz-free space over a set of the form $E\times E\subset\RR^2$, where
$$
E = \set{p(n) \,:\, n\in\NN} \subset \RR
$$
and $p$ is a polynomial, must be isomorphic to $\lipfree{\ZZ^2}$. Of course, the question is only meaningful if we assume that $p$ is non-constant. In that case, the answer is affirmative. Indeed, note that $E$ is uniformly discrete, and moreover it is not porous in $\RR$ (see Proposition \ref{pr:polynomial} below). It then follows easily from the definitions that $E\times E$ is uniformly discrete and non-porous in $\RR^2$. Thus, Theorem \ref{th:porous_lipfree} yields $\lipfree{E\times E}\sim\lipfree{\ZZ^2}$.
The same reasoning yields, more generally, that $\lipfree{E\times\stackrel{n}{\ldots}\times E}\sim\lipfree{\ZZ^n}$ whenever $E\subset\RR$ is uniformly discrete and not porous. The situation for products of different subsets of $\RR$ is more complicated, and we refer to the very recent preprint \cite{Mason} for the study of some particular cases, including an alternative solution to \cite[Question 6]{Aliaga26}.

The fact that $E$ and $E\times E$ above are not porous is well-known, and an argument is sketched in \cite[Example 4.1]{Aliaga26}, just above Question 6, for a particular case. We provide the general argument here for completeness.

\begin{proposition}\label{pr:polynomial}
Let $p$ be a non-constant polynomial with real coefficients. Then the set $E=\set{p(n):n\in\NN}$ is not porous in $\RR$.
\end{proposition}

\begin{proof}
We assume without loss of generality that the leading coefficient of $p$ is $1$, and that $p$ is strictly positive and strictly increasing in $[0,\infty)$. Let $k\geq 1$ be the degree of $p$.

Fix $\ep>0$. Since $q(n)=p(n+1)-p(n)$ is a polynomial of degree less than $k$, we may find $n_0\in\NN$ such that
$$
\frac{1}{2}n^k\leq p(n)\leq 2n^k
$$
and
$$
p(n+1)-p(n)\leq \ep n^k
$$
for all $n\geq n_0$. Choose $n_1\geq n_0$ such that $p(n_0)\leq\ep\cdot (8n_1)^k$ and consider the closed ball with center $p(n_1)$ and radius $p(n_1)$, i.e. the interval $I=[0,2p(n_1)]$. Note that
$$
2p(n_1) \leq 4n_1^k \leq \frac{1}{2}(8n_1)^k \leq p(8n_1) .
$$

Let $x\in I$. We claim that there is $n\in\NN$ such that $p(n)\in I$ and $\abs{x-p(n)}\leq \ep\cdot (8n_1)^k$. If $x<p(n_0)$, then this is satisfied by $n=n_0$ as $\abs{x-p(n_0)}\leq p(n_0)$. Otherwise, let $n$ be the largest integer such that $p(n)\leq x$; then $p(n)\in I$ and, since $x\leq 2p(n_1)\leq p(8n_1)$ and $p$ is increasing in $[0,\infty)$, we have $n_0\leq n\leq 8n_1$. Therefore
$$
\abs{x-p(n)} = x-p(n) \leq p(n+1)-p(n) \leq \ep n^k \leq \ep\cdot (8n_1)^k ,
$$
proving our claim. Therefore $I$ cannot contain a closed ball of radius at least $\ep\cdot (8n_1)^k$ that does not intersect $E\cap I$. It follows that $E$ fails the definition of porosity for all values of $\lambda$ larger than
$$
\frac{\ep\cdot(8n_1)^k}{p(n_1)} \geq \frac{\ep\cdot(8n_1)^k}{2n_1^k} \geq 4^k\ep .
$$
Since $\ep$ was arbitrary, we conclude that $E$ cannot be porous.
\end{proof}

\section*{Acknowledgments}

The author thanks Rub\'en Medina, Andr\'es Quilis and Triinu Veeorg for discussions on the topic of this note.


\end{document}